\documentclass[aop,preprint]{imsart}

\RequirePackage{amsthm,amsmath,amsfonts,amssymb}
\numberwithin{equation}{section}
\RequirePackage[authoryear]{natbib}
\RequirePackage[colorlinks,citecolor=blue,urlcolor=blue]{hyperref}
\RequirePackage{graphicx}

\startlocaldefs

\theoremstyle{plain}

\newtheorem{theorem}{Theorem}[section]
\newtheorem{lemma}[theorem]{Lemma}

\endlocaldefs

\newcommand{\be}{\begin{equation}}
\newcommand{\ee}{\end{equation}}
\newcommand{\bqa}{\begin{eqnarray}}
\newcommand{\eqa}{\end{eqnarray}}
\newcommand{\bqn}{\begin{eqnarray*}}
\newcommand{\eqn}{\end{eqnarray*}}
\newcommand{\bdes}{\begin{description}}
\newcommand{\edes}{\end{description}}
\newcommand{\bitem}{\begin{itemize}}
\newcommand{\eitem}{\end{itemize}}
\newcommand{\benum}{\begin{enumerate}}
\newcommand{\eenum}{\end{enumerate}}

\newcommand{\ep}{\varepsilon}

\newcommand{\ga}{\alpha}

\newcommand{\gl}{\lambda}

\newcommand{\gD}{\Delta}

\newcommand{\rtr}{{\rm tr}}

\newcommand{\fii}{\frac{1}{2}}
\newcommand{\fnn}{\frac1{\sqrt n}}

\newcommand{\siln}{\sum_{i=1}^n}
\newcommand{\skln}{\sum_{k=1}^n}

\newcommand{\non}{\nonumber\\}

\newcommand{\bbH}{{\bf H}}

\newcommand{\bbI}{{\bf I}}

\newcommand{\bbX}{{\bf X}}

\newcommand{\bay}{\begin{array}}
\newcommand{\eay}{\end{array}}

\begin{document}

\begin{frontmatter}
\title{A revisit of the circular law}
\runtitle{Circular Law}

\begin{aug}	
			\author[A]{\fnms{Zhidong}~\snm{Bai}\ead[label=e1]{baizd@nenu.edu.cn}}
			\author[A]{\fnms{Jiang}~\snm{Hu}\ead[label=e2]{huj156@nenu.edu.cn}}
			\address[A]{KLASMOE and School of Mathematics and Statistics, Northeast Normal University, China\printead[presep={,\ }]{e1,e2}}
		\end{aug}

\begin{abstract}
Consider a complex random $n\times n$ matrix $\bbX_n=(x_{ij})_{n\times n}$, whose entries $x_{ij}$ are independent random variables with zero means and unit variances. It is well-known that Tao and Vu  (Ann Probab 38: 2023-2065, 2010) resolved the circular law conjecture, establishing that if the $x_{ij}$'s are independent and identically distributed random variables with zero mean and unit variance, the empirical spectral distribution of $\frac{1}{\sqrt{n}}\bbX_n$ converges almost surely to the uniform distribution over the unit disk in the complex plane as $n \to \infty$. This paper demonstrates that the circular law still holds under the more general Lindeberg's condition: $$
\frac1{n^2}\sum_{i,j=1}^n\mathbb{E}|x_{ij}^2|I(|x_{ij}|>\eta\sqrt{n})\to 0,\mbox{~~~as $n \to \infty$}.
$$

This paper is a revisit of the proof procedure of the circular law by Bai in (Ann Probab 25: 494-529, 1997). The key breakthroughs in the paper are establishing a general strong law of large numbers under Lindeberg's condition and the uniform
upper bound for the integral with respect to the smallest eigenvalues of random matrices. These advancements significantly streamline and clarify the proof of the circular law, offering a more direct and simplified approach than other existing methodologies.

\end{abstract}



\end{frontmatter}

\section{Introduction}

Suppose that $\bbX_n$ is an $n
\times n$ matrix with entries $x_{kj}$, where $\{x_{kj},\
k,j=1,2,\dots\}$ forms an infinite double array of  complex
random variables.  Using the complex
eigenvalues $\lambda_1, \, \lambda_2, \, \dots, \, \lambda_{n}$ of
$\fnn \bbX_n$, we define the empirical spectral distribution  (ESD) of $\fnn \bbX_n$ by\[
\mu_n(x,y) = \frac1n \# \left\{ i \le n: \Re(\lambda_{k})\le x,\
\Im(\lambda_{k})\le y \right\},
\]
where $\# E$ denotes the cardinality of the set $E$.

The main result of our paper is the following:
\begin{theorem} \label{circularlaw} Suppose that the underlying distribution of
elements of $\bbX_n$ are independent random variables with zero means, unit variances, and satisfying the Lindeberg's condition: for any fixed $\eta>0$
\bqa\label{Lindecirc}
\frac1{n^2}\sum_{i,j=1}^n\mathbb{E}|x_{ij}^2|I(|x_{ij}|>\eta\sqrt{n})\to 0.
\eqa
Then, with probability 1, 
$\mu_n(x,y)$  tends to the circular law, i.e., the
uniform distribution over the unit disk in the complex plane.
\end{theorem}

Since the early 1950s, the conjecture that the distribution $\mu_n (x,y)$ tends to the circular law has stimulated considerable academic interest. The seminal response was proffered by \cite{Mehta67R}, addressing the scenario in which $x_{ij}$ are independently and identically distributed (i.i.d.) standard complex normal variables. \cite{Mehta67R} utilized the joint density function of the eigenvalues of the matrix $\fnn\bbX_n$, a formula originally deduced by \cite{Ginibre65S}. Silverstein in \citep{CohenK86R} and Edelman in \citep{Edelman97P} also made important contributions in Gaussian cases. For the non-Gaussian cases, \cite{Girko85C} and \cite{Bai97C} were the first to formally articulate a proof of the circular law under the conditions that $\{x_{kj},\
k,j=1,2,\dots\}$ are i.i.d. and the density of the underlying distribution is bounded. This condition was crucial in the estimation of the smallest singular values of a Hermitian type matrix. Building upon this foundation, \cite{PanZ10C} expanded the methodologies of \cite{RudelsonV08L} regarding the smallest singular values to include non-centralized matrices, subsequently negating the necessity of the density assumption and affirming the circular conjecture under conditions of $4+\varepsilon$ moments. At the same time,   \cite{TaoV08R,GotzeT10C} furthered the investigation under a lesser moment condition of $2+\varepsilon$. \cite{TaoV10Ra} subsequently resolved the conjecture by applying the Lindeberg substitution principle in the i.i.d. case. For more information about the circular law, we refer to the survey by \cite{BordenaveC12C} and references therein. This survey discussed various results around the circular law including both historical context and contemporary research directions in the field of random matrix theory.

In this paper, unlike \cite{TaoV10Ra}, we introduce a more direct and simpler method to prove the circular law, which is a non-trivial extension of \cite{Bai97C}. 
Remarkably, the novel proof prevails under a weaker condition than the i.i.d. assumption in \citep{TaoV10Ra}, specifically under Lindeberg's condition (\ref{Lindecirc}).

In the sequel, we adopt some of the notation and basic results from \citep{Bai97C,BaiS10S} to make our proof easier to follow.
Denote the ESD of the Hermitian matrix
$\bbH_n=(\fnn\bbX_n-z\bbI)(\fnn\bbX_n-z\bbI)^*$  by  $$\nu_n(x,z)=\frac1n \# \left\{1\leq i \le n: \lambda_{i}^{\bbH_n}(z)\le x \right\},$$ where $\bbI$ is the identity matrix, $\lambda_{k}^{\bbH_n}(z), 1\leq i \le n$  are the real eigenvalues of $\bbH_n$ and $^*$ stands for the complex conjugate and transpose of the matrix. The following lemma, proposed by \cite{Girko85C},  established a relation between the characteristic function of the ESD $\mu_n(x,y)$ and an integral involving the ESD of the Hermitian matrix $\bbH_n$.
\begin{lemma}\label{circlemma} For any $uv\neq 0$, we have the following formula for the two-dimensional characteristic
function $c_n(u,v)$ of $\mu_n(x,y)$, i.e.,
\bqa &&c_n(u,v)=:\int\!\!\!\int e^{iux+ivy}\mu_n(dx,dy)\non
&=&\frac{u^2+v^2}{4iu\pi}\int\!\!\!\int \frac{\partial}{\partial
s}\left[\int_0^\infty \ln x \nu_n(dx,z)\right] e^{ius+ivt}dtds,
\label{girko} \eqa where $z=s+it$ and $i=\sqrt{-1}$.
\end{lemma}
The proof of this lemma can be found in Section 12.1.2 of \cite{BaiS10S}. It is clear that to prove Theorem  \ref{circularlaw} we only need to prove for every pair $(u,v)$ such that $uv\neq 0$, $c_n(u,v)\to c(u,v)$ almost surely, where $c(u,v)$ is the characteristic function of the uniform distribution $\mu(x,y)$ of
the unit circle in the plane.   Moreover, it has been shown in Lemma 3.2  of \cite{Bai97C} that for all $uv \neq 0$,
\begin{align}\label{cuv}
c(u, v) = \frac{1}{\pi} \int_{x^2 + y^2 \leq 1} e^{iux+ivy} dxdy 
= \frac{u^2+v^2}{4iu\pi} \int  \int g(s,t)e^{ius+ivt}dt ds,
\end{align} 
where
\[ g(s,t) =
\begin{cases}
    \frac{2s}{s^2+t^2}, & \text{if } s^2+t^2 > 1,\\
    2s, & \text{otherwise}.
\end{cases}
\]
Thus, analogous to the proof steps in \cite{Bai97C}, we outline the proof of Theorem \ref{circularlaw} as follows.
\begin{itemize}
\item[Step I:]  $\nu_n(x, z)$ converges to $\nu(x, z)$ uniformly in every bounded region of $z$,
where
$\nu(x,z)$ is the limiting spectral distribution (LSD) of the sequence of matrices
$\bbH_n$, and  $\nu(x,z)$ determines the
circular law, i.e., $$
  g(s,t)=  \frac{\partial}{\partial s}\int_0^\infty \ln x\nu(dx,z).$$ 	
\item[Step II:]  
The proof of
the circular law  can be reduced to showing that, for every large $A>0$
\bqa &&
\lim_{n\to\infty}\int\!\!\!\int_{T}\left[\frac{\partial}{\partial s} \int_0^\infty
\ln x \nu_n(dx,z)-\frac{\partial}{\partial s} \int_0^\infty
\ln x \nu(dx,z)\right] e^{ius+ivt} dsdt=0, \label{limint} \eqa where
$T=\{(s,t);\ |s|\le A,|t|\le A^3\}$.
	\item[Step III:] For a suitably defined sequence $\ep_n\to0$, with probability 1,
\be
\lim_{n\to\infty}\int\!\!\!\int_{T} \left|\frac{\partial}{\partial s}\int_0^{\ep_n^2}\ln x
\nu_n(dx,z)\right|dsdt=0, \label{ctail} \ee 
\be
\lim_{n\to\infty}\int\!\!\!\int_{T} \left|\frac{\partial}{\partial s}\int_0^{\ep_n^2}\ln x
\nu(dx,z)\right|dsdt=0, \label{cctail} \ee 
and
	 \be
\lim_{n\to\infty}\int\!\!\!\int_{T(\ep)}\frac{\partial}{\partial s}\int_{0}^{\infty}
\ln x (\nu_n(dx,z)-\nu(dx,z))dsdt=0, \label{body} \ee 
where 
$T(\ep)$ is a random subset of $T$ depending on  the eigenvalues  $\lambda_k$, $k=1,\dots,n$,  and will be defined in Section \ref{SIII}. 
\end{itemize}

Compared with the pioneering work of \cite{Bai97C}, the main difficulty of the current study lies in the fact that Lindeberg's condition \ref{Lindecirc} in Theorem \ref{circularlaw} can not guarantee enough estimations of the smallest eigenvalues of $\fnn\bbX_n-z\bbI$ and the convergence rate of $\nu_n(x,z)$ to $\nu(x,z)$,
which cannot be handled in the same way as they did for the i.i.d. and higher moments existing cases. Technically, the key breakthroughs in the paper are twofold. First, we establish a general strong law of large numbers under the Lindeberg's condition, which ensures the tightness of the eigenvalues of $\fnn \bbX_n$. Second, we provide a uniform upper bound on the integral with respect to the smallest eigenvalues of $(\fnn\bbX_n-z\bbI)(\fnn\bbX_n-z\bbI)^*$, which helps the proof of the circular law without the estimates of these smallest eigenvalues. It should be noted that some uniform estimate of the smallest eigenvalues of $(\fnn\bbX_n-z\bbI)(\fnn\bbX_n-z\bbI)^*$ with respect to $z$ is crucial for other existing methods (e.g., \cite{Bai97C,PanZ10C,GotzeT10C}). 

In the following three sections, we will present the proofs of Steps I--III, respectively.

\section { Step I: Convergence of $\nu_n(x, z)$} Defined by the Stieltjes transform of $\nu(x,z)$
 $$\Delta(\alpha)=\int\frac{1}{x-\alpha}\nu(dx,z),~~~\Im\alpha>0.$$
From Theorem 2.1 of  \cite{ZhouB23L} we have that 
under Lindeberg's condition (\ref{Lindecirc}), $\nu_n(x,z)$ almost surely tends to a limit $\nu(x,z)$,  
whose Stieltjes transform \(\Delta=\Delta(\alpha)\) is the unique solution on the upper half complex plane to the equation
\begin{align}\label{eq2.1}
   \Delta = \frac{1}{\frac{|z|^2}{1 + \Delta} - (1 +  \Delta )\alpha}.
\end{align}
Here we choose $p = n$,  $\mathbf{T}_n =  \mathbf{I}$ and $\mathbf{R}_n = -\sqrt{n}z \mathbf{I}$ in Theorem 2.1 of \cite{ZhouB23L}.   From the proof of Theorem 2.1 of \cite{ZhouB23L}, we can also conclude that the convergence of $\nu_n(x,z)$ to $\nu(x,z)$ is uniformly in every bounded region of $z$. Thus, we omit the details. 
Note that \eqref{eq2.1} coincides with equation (4.2) of \cite{Bai97C}, i.e.,
\begin{align}\label{eq2.2}
  \Delta^3 + 2\Delta^2 + \left(\frac{\alpha + 1 - |z|^2}{\alpha}\right)\Delta + \frac{1}{\alpha} = 0.
\end{align}
It is clear that the solution of the above equation
has three analytic branches when $\ga\neq 0$ and when there is no
multiple root.  In fact, there is only one branch (denoted by $\Delta$) that can be the Stieltjes transform of $\nu(x,z)$. 
 Let $m_2(\ga)$ and
$m_3(\ga)$ be the other two branches.  If, say, $m_2(\ga)$ is
another Stieltjes transform.  For $\alpha$ real converging to
$-\infty$, from Vieta's formula we have  that $\gD(\ga)+m_2(\ga)+m_3(\ga)=-2$, which implies $m_3(\ga)\to-2$. But from Vieta's formula
$\alpha\gD(\ga)m_2(\ga)m_3(\ga)=-1$ and the fact that
$\alpha\gD(\ga)$ converges to $-1$ as $\alpha\to-\infty$, we would have
$m_3(\ga)$ unbounded.  Thus, it is a contradiction. Analogously, $m_3(\ga)$ is
not the Stieltjes transform of $\nu(x,z)$. 
Moreover, 
From Lemma 4.4 in  \cite{Bai97C}, we know that  
\[
    \frac{\partial}{\partial s} \int_0^\infty \ln x \nu(dx, z) = g(s, t),
\]
which together with \eqref{cuv}  defines the circular law. 
Thus, we complete the proof of Step I.

\section { Step II: Integral range reduction\label{seccirc3}}
Before we state the proof of Step II, we give two novel lemmas, which are the key to our proof.
\begin{lemma}\label{lln}
Let $x_i, i=1,2,\dots,n$ be independent nonnegative random variables with means 1. Then for any positive constant $a$ and integer $b>1$, we have that 
\bqa\label{lln1}
\mathbb{E} \siln x_iI(x_i>a)I(\#\{j;x_j>a\}<b)\le b,
\eqa
where $I(\cdot)$ is the indicator function.
\end{lemma}
\begin{proof}
Let  $\{i_1,\dots,i_k\}$ be a subset of $\{1,2,\dots,n\}$ for an integer $k$ and denote the event 
\begin{equation*}
E_{\{i_1,\dots,i_k\}}= \bigcap_{j=1}^nA_{\{j,i_1,\dots,i_k\}},
\end{equation*}
where
$$A_{\{j,i_1,\dots,i_k\}}=\left\{
    \begin{array}{ll}
      \{x_j>a\} & ~\text{if }j\in \{i_1,\dots,i_k\}\\
      \{x_j\le a\} & ~\text{otherwise}.
          \end{array}
  \right.
$$
 Without loss of generality, when the indices are listed in a subset, they are listed in their partial order, i.e., $i_1<i_2<\cdots<i_k$.  
Notice that
\begin{align}
	\text{ the events  $E_{\{i_1,\dots,i_k\}}$ are disjoint for different subscript sets.}\tag{$*$}
\end{align}
Thus, we can obtain that
\begin{align*}
S_n
=\sum_{i=1}^nx_iI(x_i>a)I(\#\{j;\ x_j>a\}<b)=\sum_{k=1}^{b-1}J_{k},
\end{align*}
where
$$J_{k} =\sum_{i_1,\dots,i_k}\sum_{t=1}^kx_{i_t} I(E_{\{i_1,\dots,i_k\}})=\max_{i_1,\dots,i_k}\sum_{t=1}^kx_{i_t} I(E_{\{i_1,\dots,i_k\}})~~~~~~~
       \mbox{(because of $(*)$)}.$$
       Define
$$
  J_{ek}= \max_{i_1,\dots,i_k}\sum_{t=1}^k\Big(\mathbb{E} x_{i_t}I(x_{i_t}>a)\Big) I\left(\bigcap_{j\neq i_t}A_{\{j,i_1,\dots,i_k\}}\right)$$and$$
   J_{ok}= \max_{i_1,\dots,i_k}\sum_{t=1}^k I\left(\bigcap_{j\neq i_t}A_{\{j,i_1,\dots,i_k\}}\right).
  $$
 It is easy to find that $\mathbb{E} S_n=\sum _{k<b}\mathbb{E} J_{k}=\sum _{k<b}\mathbb{E} J_{ek}$ and $J_{ek}\le J_{ok}$ since $\mathbb{E} x_iI(x_i\le a)\le 1$.
 Then,  because $(*)$,
 \bqa\label{lln0}
  J_{ok}
 &=&\max_{i_1,\dots,i_k}\left[\sum_{t=1}^k \left(I\big(E_{\{i_1,\dots,i_k\}}\big)+I\big(E_{\{i_1,\dots,i_k\}\backslash\{i_t\}}
 \big)\right)\right]\non
 &=&\max_{i_1,\dots,i_k}\left[kI\big(E_{\{i_1,\dots,i_k\}}\big)+\max_{t\leq k}I\big(E_{\{i_1,\dots,i_k\}\backslash\{i_t\}}
 \big)\right]\non
&\le &\sum_{i_1,\dots,i_k}kI\big(E_{\{i_1,\dots,i_k\}}
\big)+\sum_{i_1,\dots,i_{k-1}}I\big(E_{\{i_1,\dots,i_{k-1}\}}\big)\non
 &= &k\{\{\#\{i;x_i>a\}=k\}\}+\{\#\{i;x_i>a\}=k-1\}.
 \eqa 
%
 Therefore, we have
\begin{align*}
   &\mathbb{E} S_n\le \sum_{k=1}^{b-1}\mathbb{E} J_{ek}\le \sum_{k=1}^{b-1}\mathbb{E} J_{ok}\\
 \le&  (b-1) P(\#\{i; x_i>a\}\le b-1)+P(\#\{i; x_i>a\}\le b-2)\le b,
\end{align*}
which completes the proof of Lemma \ref{lln}.
\end{proof}
\begin{lemma}\label{lln*}
Let $x_{ij},1\le i,j\le n$ satisfying the conditions in Theorem \ref{circularlaw}, then we have the following Strong Law of Large Numbers 
\bqa
\frac1{n^2}\sum_{i,j=1}^n |x_{ij}^2|\to 1, a.s.. 
\eqa
\end{lemma}
\begin{proof}
Define 
\bqn
S_{n1}&=&\frac1{n^2}\sum_{i,j=1}^n|x_{ij}^2|I(|x_{ij}|>n^{3/4}),\\
S_{n2}&=&\frac1{n^2}\sum_{i,j=1}^n|x_{ij}^2|I(|x_{ij}|\le n^{3/4}),
\eqn
and denote the number of pairs $(ij)$ such that $|x_{ij}|>n^{3/4}$ by $N$. Then by Lemma \ref{lln}, we have
\bqn
\mathbb{E} S_{n1}I(N\le n^{3/4})\le n^{-5/4},
\eqn
which together with the Borel-Cantelli lemma implies
$$S_{n1}I(N\le n^{3/4})\to 0, ~a.s..$$ 
Moreover, by Bernstein's inequality, we have that
\begin{align*}
  &\mathbb{P}(S_{n1}I(N>n^{3/4})\ne 0)
\le \mathbb{P}\Big(\sum_{i,j=1}^nI(|x_{ij}|> n^{3/4})> n^{3/4}\Big)\\
=&\mathbb{P}\Big(\sum_{i,j=1}^n[I(|x_{ij}|> n^{3/4})-\mathbb{E}I(|x_{ij}|> n^{3/4})]> n^{3/4}-\sum_{i,j=1}^n\mathbb{E}I(|x_{ij}|> n^{3/4})]\Big)
\le 2e^{-bn^{3/4}}
\end{align*}

for some constant $b>0$. Here, we have used the facts that 
\bqn
&&Var(\sum_{i,j=1}^nI(|x_{ij}|> n^{3/4})\le \mathbb{E}(\sum_{i,j=1}^nI(|x_{ij}|> n^{3/4})\\
&\le& \frac1{n^{3/2}}\sum_{i,j=1}^n \mathbb{E}|x_{ij}^2|I(|x_{ij}|>n^{3/4})\le \frac1{n^{3/2}}\sum_{i,j=1}^n \mathbb{E}|x_{ij}^2|I(|x_{ij}|>n^{1/2})=o(n^{1/2}). 
\eqn 
Thus, we have that $$S_{n1}\to 0, ~a.s..$$
For $S_{n2}$, we have 
\bqa
0\le 1-\mathbb{E}S_{n2}\le \frac1{n^2}
\sum_{i,j=1}^n\mathbb{E}|x_{ij}^2|I(|x_{ij}|>n^{3/4})\to 0
\eqa
and 
\bqa
\mathbb{E}\left[S_{n2}-\mathbb{E} S_{n2}\right]^6
\le O(n^{-3/2}).
\eqa
From the two estimates above, we have $\frac1{n^2}\sum_{i,j=1}^n |x_{ij}^2|I(|x_{ij}|\le n^{3/4})\to 1, a.s.$. Then the proof is complete.
\end{proof}

Next we present the proof of Step II.  
Let \[
g_n(s, t) = \frac{\partial}{\partial s} \int_0^\infty \log x \, \nu_n(dx, z)\mbox{~~and~~}g(s, t) = \frac{\partial}{\partial s} \int_0^\infty \log x \, \nu(dx, z).
\]
It has been shown in Lemma 11.7 in \cite{BaiS10S} that
for any $uv\neq 0$ and $A>2$,  \be \left|\int_{|s|\ge
A}\int_{-\infty}^{\infty}g_n(s,t)e^{ius+ivt}dtds\right| \le
\frac{4\pi}{|v|}e^{-\fii|v|A}+\frac{2\pi}{n|v|}\sum_{k=1}^nI\left(|\gl_k|\ge\frac12
A\right) \label{lm431} \ee and \be \left|\int_{|s|\le A}\int_{|t|\ge
A^3}g_n(s,t)e^{ius+ivt}dsdt\right| \le \frac{8A}{A^2-1}+\frac{4\pi
A}{n}\skln I\left(|\gl_k|\ge A\right). \label{lm432} \ee
If the function $g_n(s, t)$ is replaced by $g(s, t)$, the two inequalities above hold without the second terms. 
 It is worth noting that these results are obtained by basic calculations and are independent of Theorem \ref{circularlaw}'s moment conditions. Thus, from Lemma \ref{lln*}, one can find that the right-hand sides of
(\ref{lm431}) and (\ref{lm432}) can be made arbitrarily small by
making $A$ large enough because of the fact
\bqa\label{lln3}
\limsup_n \frac1n\skln I(|\gl_k|\ge A)\le \limsup_n \frac1{n^2A^2}\rtr \bbX_n\bbX_n^*\le \frac1{A^2}, a.s..
\eqa
The same conclusion holds when $g_n(s,t)$ is
replaced by $g(s,t)$. Define set
\[
T=\{(s,t):\ |s|\le A, |t|\le A^3 \}.
\]
Therefore,  the proof of the circular law 
reduces to show that, for any large $A>0$, \be
\int\!\!\!\int_{T}[g_n(s,t)-g(s,t)]e^{ius+ivt}dsdt \to 0, ~~\mbox{a.s.
}. \label{trntd} \ee

\section {Step III: Completing  of the proof Theorem \ref{circularlaw}}\label{SIII}
In this section, we start with a lemma to complete the proof of Theorem \ref{circularlaw}.  
\begin{lemma}\label{limittozero}
	Let $f_n(\varepsilon)$ be a sequence of functions of $\varepsilon$. If for every fixed $\varepsilon>0$, $f_n(\varepsilon) \to 0$ as $n \to \infty$, then there exists a sequence $\varepsilon_n \to 0$ such that $f_n(\varepsilon_n) \to 0$.
\end{lemma}
\begin{proof}
	By the assumption, for every $k \geq 1$, $f_n(1/k) \to 0$. Therefore, we can find $n_k$ such that for all $n \geq n_k$, $f_n(1/k) < 1/k$. When $n_k \leq n < n_{k+1}$, define $\varepsilon_n = 1/k$. Then $\varepsilon_n \to 0$ and $f_n(\varepsilon_n) < 1/k \to 0$. The proof is complete.
\end{proof} 
Now,
we define the random subset of $\mathbb R^2$
\bqa\label{randt}
T(\ep)=T\backslash [\cup_{k=1}^n\{(s,t):|t-\gl_{ki}|\le \ep/2, |z-\gl_k|\le \ep\}].
\eqa 
Intuitively, $T(\ep)$ is $T$ excluding the drum shaped regions $\{(s,t):|t-\gl_{ki}|\le \ep/2, |z-\gl_k|\le \ep\}$ around all complex eigenvalues  $\gl_k$, $k=1,\dots,n$.
For any $(s,t)\in T(\ep)$ and any $k\le n$, it is easy to check 
\[ |\gl_k-z|\ge
\begin{cases}
    \ep, & \text{if~} |\gl_{ki}-t|\le \ep/2\\
     \ep/2, &  \text{otherwise}.
\end{cases}
\]
On the other hand, by the arbitrary of $A$ and Lemma \ref{lln*}, we can assume that $|\gl_k-z|$ is uniformly bounded from above for  $(s,t)\in T$.

From Step II, to complete the proof of Theorem \ref{circularlaw}, we only need to prove \eqref{limint}, i.e.,
\bqa &&
\lim_{n\to\infty}\int\!\!\!\int_{T}\left[\frac{\partial}{\partial s} \int_0^\infty
\ln x \nu_n(dx,z)-\frac{\partial}{\partial s} \int_0^\infty
\ln x \nu(dx,z)\right] e^{ius+ivt} dsdt=0.\nonumber \eqa
Notice that 
$$\left|\int\!\!\!\int_{T\backslash T(\ep)}\frac{\partial}{\partial s} \int_0^\infty
\ln x \nu_n(dx,z)e^{ius+ivt} dsdt\right|\leq \int\!\!\!\int_{T} \left|\frac{\partial}{\partial s} \int^{\ep^2}_0
\ln x \nu_n(dx,z) \right|dsdt$$
and
$$\left|\int\!\!\!\int_{T\backslash T(\ep)}\frac{\partial}{\partial s} \int_0^\infty
\ln x \nu(dx,z)e^{ius+ivt} dsdt\right|\leq \int\!\!\!\int_{T} \left|\frac{\partial}{\partial s} \int^{\ep^2}_0
\ln x \nu(dx,z) \right|dsdt.$$
Thus, 
by Lemma \ref{limittozero}, what we need is proving Step III, i.e.,
\bqa
&\int\!\!\!\int_{T}\left|\frac{\partial}{\partial s}\int_0^{\ep_n^2}
 \log x \nu_n(dx,z)\right| dtds\stackrel{a.s.} \to 0, ~~\mbox{as $n\to\infty$},
\label{key2}\\
&\int\!\!\!\int_{T}\left|\frac{\partial}{\partial s}\int_0^{\ep_n^2}
 \log x \nu(dx,z)\right|  dtds\to 0, ~~\mbox{as $n\to\infty$},
\label{key3}\\
 &\int\!\!\!\int_{T(\ep)}\frac{\partial}{\partial s}\int_{0}^{\infty}
\ln x (\nu_n(dx,z)-\nu(dx,z))dsdt\stackrel{a.s.} \to 0.~~\mbox{as $n\to\infty$}.\label{key1}
\eqa 
First, we show (\ref{key2}). Recall  the complex
eigenvalues $\lambda_k=\lambda_{kr}+i\lambda_{ki}$ of
$\fnn \bbX_n$. 
Note that 
\[
\frac{1}{n} \sum_{k=1}^{n} \frac{\lambda_{kr} - s}{|\lambda_k - z|^2}
= -\frac{1}{2n} \sum_{k=1}^{n} \frac{\partial}{\partial s} \log(|\lambda_k - z|^2)
= -\frac{1}{2} \frac{\partial}{\partial s} \int_{0}^{\infty} \log x \, \nu_n(dx, z).
\]
Thus, we have that
%
\begin{gather*}
	\int\!\!\!\int_{T}\left|\frac{\partial}{\partial s}\int_0^{\ep_n^2}
 \log x \nu_n(dx,z)\right| dtds
 =\int\!\!\!\int_{T}\left|\frac1{n}\sum_{k=1}^n\frac{2(\gl_{kr}-s)I_k(\ep_n)}{(\gl_{kr}-s)^2+(\gl_{ki}-t)^2}\right|dsdt\non
 \le \frac2{n}\sum_{k=1}^n\int\!\!\!\int_{T}\frac{|\gl_{kr}-s|I_k(\ep_n)}{(\gl_{kr}-s)^2+(\gl_{ki}-t)^2}dsdt,
\end{gather*}
where 
$I_k(\ep_n)=I((\gl_{kr}-s)^2+(\gl_{ki}-t)^2\leq\ep_n^2)$.
Note that on $\{(\gl_{kr}-s)^2+(\gl_{ki}-t)^2\leq\ep_n^2\}$, we have $|z-\gl_k|\le \ep_n$. Therefore, 
\bqa\label{key4}
 &&\int\!\!\!\int_{T}\left|\frac{\partial}{\partial s}\int_0^{\ep_n^2}
 \log x \nu_n(dx,z)\right| dtds \non
 &\le&\frac2{n}\sum_{k=1}^n\int_{|\gl_k-z|\le \ep_n}\frac{|\gl_{kr}-s|}{(\gl_{kr}-s)^2+(\gl_{ki}-t)^2}dsdt= 8\ep_n \to 0,
 \eqa
which proves (\ref{key2}). 

Similar to (\ref{key4}), one can prove that 
\bqn
 &&\int\!\!\!\int_{T}\left|\frac{\partial}{\partial s}\int_0^{\ep_n^2}
 \log x \nu(dx,z)\right|  dtds\non
 &\le&\frac{2}{\pi}\int_{|\gl|\le 1} \int_{|\gl-z|\le \ep_n}\frac{|\gl_r-s|}{(\gl_{r}-s)^2+(\gl_{i}-t)^2}dsdtd\gl_rd\gl_i= 8 \ep_n \to 0.
\eqn
This proves (\ref{key3}). 

For \eqref{key1}, for any $t\in (-A^3, A^3)$ and $\ep>0$, it is easy to find that the intersection of $T(\ep)$ with straight line $t$ consists of finite intervals. Write the end points as $s_1,\dots, s_{2\ell(t)}$ and 
$z_j=(s_j, t)$.  It is worth noting that here $\ell(t)$ can be 0.
It follows that 
\bqa
&&\int\!\!\!\int_{T(\ep)}\frac{\partial}{\partial s}\int_{0}^\infty
 \log x \nu_n(dx,z)dtds\non
 &= &\frac1n\skln\int_{-A^3}^{A^3}\sum_{j=1}^{\ell(t)}\int_{s_{2j-1}}^{s_{2j}}
\frac{2(s-\gl_{kr})}{(s-\gl_{kr})^2+(t-\gl_{ki})^2}
dsdt\non
 &=&\frac1n\skln\int_{-A^3}^{A^3}\sum_{j=1}^{\ell(t)}(\log(|z_{2j}-\gl_k|^2)-\log(|z_{2j-1}-\gl_k|^2)dt\non
 &=&\int_{-A^3}^{A^3}\sum_{j=1}^{\ell(t)}\left(\int_{0}^\infty
 \log x \nu_n(dx,z_{2j})- \log x \nu_n(dx,z_{2j-1})\right)dt\nonumber.  \eqa
   It is shown by \cite{ZhouB23L} that $\nu_n(x,z)$ almost surely and uniformly tends to $\nu(x,z)$ and $\nu(\cdot,z)$ in every bounded region of $z$. Thus,  by (\ref{randt}) and the dominated convergence theorem, we have that
 \begin{align*}
&\int\!\!\!\int_{T(\ep)}\frac{\partial}{\partial s}\int_{0}^\infty
 \log x \nu_n(dx,z)dtds-\int\!\!\!\int_{T(\ep)}\frac{\partial}{\partial s}\int_{0}^\infty
 \log x \nu(dx,z)dtds\\
 =&\int_{-A^3}^{A^3}\sum_{j=1}^{\ell(t)}\int_{0}^\infty
 \log x \Big(\nu_n(dx,z_{2j})- \nu(dx,z_{2j})-\nu_n(dx,z_{2j-1})
+\nu(dx,z_{2j-1})\Big)dt\non
 \stackrel{a.s.} \to&0.
\end{align*}
  Then, we obtain \eqref{key1}.  
  Consequently, Theorem \ref{circularlaw} is proved.
%
%
%
%
%
%
%
%
%
%
	
	\section*{Acknowledgments}
	Zhidong Bai was partially supported by NSFC Grants No.12171198, No.12271536, and Team Project of Jilin Provincial Department of Science and Technology No.20210101147JC. Jiang Hu was partially supported by NSFC Grants Nos. 12171078, 12292980, No12292982, No. National Key R $\&$ D Program of China No. 2020YFA0714102, and Fundamental Research Funds for the Central Universities No. 2412023YQ003.

\begin{thebibliography}{14}

\bibitem[\protect\citeauthoryear{Bai}{1997}]{Bai97C}
\begin{barticle}[author]
\bauthor{\bsnm{Bai},~\bfnm{Zhi~Dong}\binits{Z.~D.}}
(\byear{1997}).
\btitle{Circular Law}.
\bjournal{The Annals of Probability}
\bvolume{25}
\bpages{494--529}.
\end{barticle}
\endbibitem

\bibitem[\protect\citeauthoryear{Bai and Silverstein}{2010}]{BaiS10S}
\begin{bbook}[author]
\bauthor{\bsnm{Bai},~\bfnm{Zhi~Dong}\binits{Z.~D.}} \AND
  \bauthor{\bsnm{Silverstein},~\bfnm{Jack~W.}\binits{J.~W.}}
(\byear{2010}).
\btitle{Spectral Analysis of Large Dimensional Random Matrices. {{Second
  Edition}}}.
\bpublisher{Springer Verlag}.
\end{bbook}
\endbibitem

\bibitem[\protect\citeauthoryear{Bordenave and
  Chafa{\"i}}{2012}]{BordenaveC12C}
\begin{barticle}[author]
\bauthor{\bsnm{Bordenave},~\bfnm{Charles}\binits{C.}} \AND
  \bauthor{\bsnm{Chafa{\"i}},~\bfnm{Djalil}\binits{D.}}
(\byear{2012}).
\btitle{Around the Circular Law}.
\bjournal{Probability Surveys}
\bvolume{9}
\bpages{1--89}.
\end{barticle}
\endbibitem

\bibitem[\protect\citeauthoryear{Cohen, Kesten and Newman}{1986}]{CohenK86R}
\begin{bbook}[author]
\beditor{\bsnm{Cohen},~\bfnm{Joel~E.}\binits{J.~E.}},
  \beditor{\bsnm{Kesten},~\bfnm{Harry}\binits{H.}} \AND
  \beditor{\bsnm{Newman},~\bfnm{Charles~M.}\binits{C.~M.}}, eds.
(\byear{1986}).
\btitle{Random Matrices and Their Applications}.
\bseries{Contemporary Mathematics}
\bvolume{v. 50}.
\bpublisher{American Mathematical Society}.
\end{bbook}
\endbibitem

\bibitem[\protect\citeauthoryear{Edelman}{1997}]{Edelman97P}
\begin{barticle}[author]
\bauthor{\bsnm{Edelman},~\bfnm{Alan}\binits{A.}}
(\byear{1997}).
\btitle{The Probability That a Random Real {{Gaussian}} Matrix Has \$k\$ Real
  Eigenvalues, Related Distributions, and the Circular Law}.
\bjournal{Journal of Multivariate Analysis}
\bvolume{60}
\bpages{203--232}.
\end{barticle}
\endbibitem

\bibitem[\protect\citeauthoryear{Ginibre}{1965}]{Ginibre65S}
\begin{barticle}[author]
\bauthor{\bsnm{Ginibre},~\bfnm{Jean}\binits{J.}}
(\byear{1965}).
\btitle{Statistical {{Ensembles}} of {{Complex}}, {{Quaternion}}, and {{Real
  Matrices}}}.
\bjournal{Journal of Mathematical Physics}
\bvolume{6}
\bpages{440--449}.
\end{barticle}
\endbibitem

\bibitem[\protect\citeauthoryear{Girko}{1985}]{Girko85C}
\begin{barticle}[author]
\bauthor{\bsnm{Girko},~\bfnm{V.~L.}\binits{V.~L.}}
(\byear{1985}).
\btitle{Circular {{Law}}}.
\bjournal{Theory of Probability \& Its Applications}
\bvolume{29}
\bpages{694--706}.
\end{barticle}
\endbibitem

\bibitem[\protect\citeauthoryear{G{\"o}tze and Tikhomirov}{2010}]{GotzeT10C}
\begin{barticle}[author]
\bauthor{\bsnm{G{\"o}tze},~\bfnm{Friedrich}\binits{F.}} \AND
  \bauthor{\bsnm{Tikhomirov},~\bfnm{Alexander}\binits{A.}}
(\byear{2010}).
\btitle{The Circular Law for Random Matrices}.
\bjournal{The Annals of Probability}
\bvolume{38}.
\end{barticle}
\endbibitem

\bibitem[\protect\citeauthoryear{Mehta}{1967}]{Mehta67R}
\begin{bbook}[author]
\bauthor{\bsnm{Mehta},~\bfnm{M.~L.}\binits{M.~L.}}
(\byear{1967}).
\btitle{Random {{Matrices}} and the {{Statistical Theory}} of {{Energy
  Levels}}}.
\bpublisher{Academic Press}.
\end{bbook}
\endbibitem

\bibitem[\protect\citeauthoryear{Pan and Zhou}{2010}]{PanZ10C}
\begin{barticle}[author]
\bauthor{\bsnm{Pan},~\bfnm{Guangming}\binits{G.}} \AND
  \bauthor{\bsnm{Zhou},~\bfnm{Wang}\binits{W.}}
(\byear{2010}).
\btitle{Circular Law, Extreme Singular Values and Potential Theory}.
\bjournal{Journal of Multivariate Analysis}
\bvolume{101}
\bpages{645--656}.
\end{barticle}
\endbibitem

\bibitem[\protect\citeauthoryear{Rudelson and Vershynin}{2008}]{RudelsonV08L}
\begin{barticle}[author]
\bauthor{\bsnm{Rudelson},~\bfnm{Mark}\binits{M.}} \AND
  \bauthor{\bsnm{Vershynin},~\bfnm{Roman}\binits{R.}}
(\byear{2008}).
\btitle{The Least Singular Value of a Random Square Matrix Is
  \${{O}}(N{\textasciicircum}\{-1/2\})\$}.
\bjournal{Comptes Rendus Mathematique}
\bvolume{346}
\bpages{893--896}.
\end{barticle}
\endbibitem

\bibitem[\protect\citeauthoryear{Tao and Vu}{2008}]{TaoV08R}
\begin{barticle}[author]
\bauthor{\bsnm{Tao},~\bfnm{Terence}\binits{T.}} \AND
  \bauthor{\bsnm{Vu},~\bfnm{Van}\binits{V.}}
(\byear{2008}).
\btitle{Random Matrices: The Circular Law}.
\bjournal{Communications in Contemporary Mathematics}
\bvolume{10}
\bpages{261--307}.
\end{barticle}
\endbibitem

\bibitem[\protect\citeauthoryear{Tao and Vu}{2010}]{TaoV10Ra}
\begin{barticle}[author]
\bauthor{\bsnm{Tao},~\bfnm{Terence}\binits{T.}} \AND
  \bauthor{\bsnm{Vu},~\bfnm{Van}\binits{V.}}
(\byear{2010}).
\btitle{Random Matrices: {{Universality}} of {{ESDs}} and the Circular Law}.
\bjournal{The Annals of Probability}
\bvolume{38}
\bpages{2023--2065}.
\end{barticle}
\endbibitem

\bibitem[\protect\citeauthoryear{Zhou, Bai and Hu}{2023}]{ZhouB23L}
\begin{barticle}[author]
\bauthor{\bsnm{Zhou},~\bfnm{Huanchao}\binits{H.}},
  \bauthor{\bsnm{Bai},~\bfnm{Zhidong}\binits{Z.}} \AND
  \bauthor{\bsnm{Hu},~\bfnm{Jiang}\binits{J.}}
(\byear{2023}).
\btitle{The {{Limiting Spectral Distribution}} of {{Large-Dimensional General
  Information-Plus-Noise-Type Matrices}}}.
\bjournal{Journal of Theoretical Probability}
\bvolume{36}
\bpages{1203--1226}.
\end{barticle}
\endbibitem

\end{thebibliography}

\end{document}